\documentclass{article}
\usepackage[margin=1in]{geometry}
\usepackage{graphicx} 
\usepackage{todonotes}
\usepackage{changes}
\usepackage{amsmath}
\usepackage{amssymb}
\usepackage{amsthm}
\usepackage[norelsize, linesnumbered, ruled, lined, boxed, commentsnumbered]{algorithm2e} 
\usepackage{optidef}
\usepackage{floatflt}

\usepackage{float}
\usepackage{hyperref} 
\usepackage{enumitem} 

\newtheorem{proposition}{Proposition}[section]   
\theoremstyle{remark}
\newtheorem*{remark}{Remark}        
\newcommand{\sat}{SAT}
\newcommand{\satCall}{\sat{} call}
\newcommand{\maxsat}{MaxSAT}

\newcommand{\unsat}{UNSAT}
\newcommand{\cnfopt}{CNFOPT}
\newcommand{\cnfoptProblem}{\cnfopt{} problem}
\newcommand{\baseMethod}{Integral Clause Cuts Algorithm}
\newcommand{\baseMethodShort}{ICCA}
\newcommand{\proofMethod}{Learned Clause Cuts Algorithm}
\newcommand{\proofMethodShort}{LCCA}
\newcommand{\rcTwo}{RC2 Stratified}
\newcommand{\satSolver}{\sat{} solver}
\newcommand{\maxsatSolver}{\maxsat{} solver}
\newcommand{\satProblem}{\sat{} problem}
\newcommand{\maxsatProblem}{\maxsat{} problem}
\newcommand{\milp}{MILP}
\newcommand{\milpLong}{mixed-integer linear programming}

\newcommand{\milpSolver}{\milp{} solver}
\newcommand{\pysat}{PySAT}
\newcommand{\clauseCut}{Clause Cut}
\newcommand{\gurobi}{Gurobi}
\newcommand{\lpRelaxation}{LP relaxation}
\newcommand{\clauseInequality}{clause inequality}
\newcommand{\clauseInequalities}{clause inequalities}
\newcommand{\noGoodCut}{no-good cut}

\title{Introducing \clauseCut{}s: Strong No-Good Cuts for MaxSAT Problems in Mixed Integer Linear Programming}

\author{
    \textbf{Max Engelhardt}\footnote{max.engelhardt@utn.de},
    \textbf{Milan Adhikari},
    \textbf{Jonasz Staszek},
    \textbf{Alexander Martin}\\[4pt]
    \small
    \textit{Department of Liberal Arts and Sciences, University of Technology Nuremberg}
}

\date{August 2025}

\begin{document}

\maketitle
\begin{abstract}
    \noindent In this paper we introduce \clauseCut{}s: linear inequalities obtained from clauses that are logically implied by a CNF formula, resembling strengthened \noGoodCut{}s.
    With these cuts, we tighten \milpLong{} (\milp) formulations of random weighted partial \maxsatProblem{}s, which have remained particularly challenging for core-guided complete \maxsatSolver{}s. 
    Our approaches treat variables that attain integral values at the \lpRelaxation{} as partial assignments which are supplied to a \satSolver{} as assumptions.
    When infeasible, these assignments are ruled out with \clauseCut{}s which are further strengthened with the \satSolver{}.
    Two separation algorithms are introduced, one that utilizes a \sat-oracle and finds \clauseCut{}s in the set of variables with integral values, and another that uses the clauses learned by a conflict driven clause learning (CDCL) \satSolver{} while evaluating the partial assignment.
    Experiments on SATLIB benchmarks demonstrate substantial performance gains of up to two orders of magnitude compared to the general purpose \milpSolver{} \gurobi{} 12, taking only 7.8\% of \gurobi{}'s runtime for the whole problem set.
    Results also surpass the specialized \maxsatSolver{} RC2, taking only 60\% of its runtime.
    In some cases, our optimization takes only slightly longer than a single \satCall{} on the \sat{}-formula.
    We explain the source of these gains and the limitations of standard MILP formulations in this context.    
\end{abstract}

\section{Introduction}
The Boolean Satisfiability (\sat) problem, a famous NP-complete problem \cite{cook_complexity_1971}, determines whether a given propositional formula is satisfiable.
Powerful \satSolver{}s exist, which can solve many \satProblem{}s impressively fast in practice. 
Most current state-of-the-art \satSolver{}s are conflict driven clause learning (CDCL) \satSolver{}s; an overview of the workings of CDCL \satSolver{}s is given in \cite{biere_handbook_2021}.

In applications, satisfying all the constraints is not always feasible, nonetheless it is 
desirable to satisfy the maximum number of constraints.
Such a problem is called Maximum Satisfiability (\maxsat), and constitutes an optimization extension of \sat{} that asks for a variable assignment that maximizes the number of satisfied clauses in a boolean formula. 
As shown by~\cite{krentel_complexity_1986}, the 
\maxsatProblem{} is NP-hard and even OptP-complete (optimize functions that can be evaluated in polynomial time).
By completeness, all OptP problems can be encoded in \maxsat. 
As a result, \maxsatSolver{}s are powerful tools with diverse applications in planning, scheduling, data analysis, correlation clustering, upgrading software packages, electronic design automation, design debugging, auctions, bioinformatics, course timetabling, routing, as well as error localization \cite{bacchus_maximum_2021}.

\maxsat{} may be solved using exact (complete) or inexact (incomplete, heuristic) approaches. 
Numerous inexact approaches to solve \maxsat{} have been developed over the years \cite{ansotegui_incomplete_2022, berg_core-boosted_2019, demirovic_techniques_2019, kemppainen_incomplete_2020}.
However, the focus of this paper will be on exact algorithms that guarantee to find a globally optimal solution.
Such exact algorithms have exploited branch-and-bound approaches \cite{heras_minimaxsat_2008, li_combining_2021,li_boosting_2022},
core-guided approaches \cite{alviano_maxsat_2015, ansotegui_sat-based_2013, ansotegui_wpm3_2017, ansotegui_exploiting_2016, fu_solving_2006, heras_core-guided_2011, ignatiev_rc2_2019, marques-silva_algorithms_2008, morgado_core-guided_2014},
implicit hitting set (IHS) approaches \cite{ lee_solving_2011, hutchison_postponing_2013, saikko_lmhs_2016},
or objective bounding approaches \cite{een_translating_2006, ignatiev_progression_2014, koshimura_qmaxsat_2012}.
A conceptually similar approach to \clauseCut{}s has been taken by Bockmayr and Pisaruk in \cite{bockmayr_detecting_2006}, who use \noGoodCut{}s for monotone constraints. 

In this paper, we propose to use \satSolver{}s to find unsatisfiable cores that are used for the derivation of linear inequalities, which strengthen a \milpLong{} (\milp{}) formulation of the \maxsatProblem{}.
We call these linear inequalities \clauseCut{}s, as they correspond to clauses implied by the \sat{}-formula.
We provide two separation algorithms (Section \ref{sec:algorithms}) to find \clauseCut{}s that cut off a given fractional point.
To the best of our knowledge, the methods proposed in this paper are the first ones to separate general fractional points with strong \noGoodCut{}s in the context of \milp{}s constrained by CNF-formulae.

We present computational studies comparing our approaches against the general purpose \milpSolver{} \gurobi{} 12 \cite{gurobi_optimization_llc_gurobi_2024} and the \maxsatSolver{} \rcTwo{} \cite{ignatiev_rc2_2019}. 
As previous empirical studies indicate that randomly generated \maxsatProblem{}s are particularly challenging for core-guided \maxsatSolver{}s \cite{liffiton2009generalizing}, we use random instances from the SATLIB benchmark library \cite{hoos_satlib_nodate}, with randomly generated integral linear objective functions, as benchmarks.

Prior to its removal from the \maxsat{} Evaluations in 2017, incomplete stochastic local search solvers consistently achieved stronger performance in the ``random'' category \cite{cai2016new, liu2017should}. 
The \maxsat{} Evaluations' organizing committee itself acknowledged that the decline in participation of incomplete solvers in 2017 was largely due to the removal of the category of ``random'' benchmarks \cite{maxsat_evaluations_2017_slides}, suggesting that these instances were precisely where incomplete approaches demonstrated their competitive advantage.
An evaluation using random problems thus provides valuable insight into the impact of \milp{} strengthening in a \maxsat{} context.
On these random benchmarks, our algorithms outperform \gurobi{} by up to two orders of magnitude, and one of them is significantly faster than \rcTwo{}.
We further contribute explanations on why our methods perform so well and what makes the randomly generated instances easy to solve.

\textbf{Contents:} This paper is composed of six sections.
Following this introduction, Section \ref{sec:Basics} focuses on the theoretical basics and foundations of \maxsat{} and its embedding into \milp{} and proves results necessary for the algorithms. 
Section \ref{sec:algorithms} presents our \sat{}‐based algorithms for separating and strengthening \clauseCut{}s. 
In Section \ref{sec:study}, we detail the computational study including experimental setup and performance comparisons. 
Finally, Section \ref{sec:conclusion} concludes the paper. 

\section{Theoretical Basics}
\label{sec:Basics}
In this section, we introduce the basic terminology and definitions necessary for readers unfamiliar with \sat{} and \milp{}. 
We formally define literals, clauses, and CNF formulas, along with the standard notions of assignments and satisfiability, closely to \cite{ansotegui_solving_2013} and \cite{li_combining_2021}.
We then present the \maxsatProblem{} and its weighted and partial variants, which serve as the paper's optimization targets.

Building on this foundation, we embed \maxsatProblem{}s into the \milp{} framework by translating logical clauses into linear inequalities.
As the main problem of interest in this paper, we introduce the CNFOPT formulation of \maxsat{} as an ILP whose feasible set is formed by the solutions of a \sat-formula and show that it is equivalent to \maxsat{}. 
We then show how clauses can be embedded as linear inequalities, and how \satSolver{}s, particularly those based on \textit{conflict-driven clause learning} (CDCL), can be employed to verify implied clauses via assumptions-based \sat-solving.
This enables the generation of \textit{\clauseCut{}s} to strengthen the \lpRelaxation{}, and is needed to consequently build the two separation algorithms in Section \ref{sec:algorithms} and understand their performance.
We conclude in Section \ref{sec:supportingHyperplanes}, with a proof of the quality of the cuts generated by our procedures, showing that minimal unsatisfiable cores correspond to supporting hyperplanes.

\subsection{Logical foundations of \sat{} and \maxsat{}}
\label{sec:Basics:Logical}
\paragraph{}
\noindent We define a \textit{literal} $l$ as a binary variable $x$ or its negation $\neg x=1-x$.
A \textit{clause} $C$ is a disjunction of $k$ literals $l_1 \lor ... \lor l_k$.
If we want to emphasize that a disjunction or a single literal is to be seen as a clause, we put it in parentheses, like $C=(\neg x)$.
A propositional formula $\varphi$ is in \textit{Conjunctive Normal Form} (CNF), if it is a conjunction of $m$ clauses $C_1 \land ... \land C_m$.
We denote the Boolean set with $\mathbb{B}:=\{0,1\}$.
The set of binary variables appearing in $\varphi$ is denoted by $\text{var}(\varphi)$.
An \textit{assignment} is a function $a: \text{var}(\varphi) \to \mathbb{B}$, a \textit{partial assignment} is a function $b: S \to \mathbb{B}$, where $S \subset \text{var}(\varphi)$.
As an abbreviation, a (partial) assignment can also be noted by the conjunction of the literals it fixes.
For example, $a: \; x \mapsto 1, \; y \mapsto 0$ can be written as $x \land \neg y$.
The literal $x$ is \textit{satisfied} by an assignment $a$ if and only if $a(x)=1$; if $a(x)=0$, then the literal $\neg x$ is satisfied.
A clause is satisfied exactly if at least one of its literals is satisfied.
A CNF is satisfied exactly if all its clauses are satisfied.
We say that a clause $C$ is \textit{implied} by the CNF $\varphi$ (``$\varphi \implies C$'') exactly if all $a$ that satisfy $\varphi$ will also satisfy $C$.
The \textit{\satProblem{}} is the decision problem if, for a given CNF $\varphi$, there exists an assignment that satisfies $\varphi$.
The \textit{\maxsatProblem{}} is the optimization version of \sat, where a solution is sought that has a minimal number of unsatisfied clauses. 
In a \textit{weighted} \maxsatProblem{}, a positive integer weight is assigned to the clauses in the CNF. 
An assignment's objective value is then given by the sum of the weights of the unsatisified clauses.
The weighted \maxsatProblem{} then asks for an assignment with minimal objective. 
In the \textit{partial} \maxsatProblem{}, the CNF's clauses are divided into two sets, hard and soft clauses.
The \textit{hard clauses} must be satisfied by all solutions.
The \textit{soft clauses} may be violated and hence determine the objective value of the solution in the same fashion. 
In the following, we will embed these logical formulations into ILP.
For a general overview and introduction to (mixed) integer linear programming, we refer, for instance, to Nemhauser et al.~\cite{nemhauser_integer_1988}.

\subsection{CNFOPT as an ILP and \clauseCut{}s}
\label{sec:Basics:CNFasILP}
With \textit{(Mixed) Integer Linear Programs} we mean problems of the form 
\begin{equation}
    \underset{x}{\min}\;\; c^Tx \;\; \text{s.t.} \;\; Ax\geq b , \;\; \forall i \in I: x_i \in \mathbb{Z} ,
\end{equation}
where $c$ and $b$ are appropriately sized rational vectors and $A$ is an appropriately sized rational matrix, while $I$ is an index set of the components of $x$ which must be integral.
If all variables $x_i$ need to be integral, the problem is called an \textit{integer linear program}, or ILP for short.
We denote the polyhedron that is the convex hull of the \milp{}'s feasible points with $P_I:=\text{conv}(\{x \; \mid \; Ax\geq b , \;\; \forall i \in I: x_i \in \mathbb{Z}\}) $.
Furthermore, for any given MILP, we define its \textit{\lpRelaxation{}} -- or simply relaxation -- as the linear program obtained by removing the integrality constraints from the MILP:
\begin{equation}
    \underset{x}{\min}\;\; c^Tx \;\; \text{s.t.} \;\; Ax\geq b .
\end{equation}
The polytope of the convex hull of the relaxation's feasible points is called $P_R:=\{x \mid Ax\geq b \}$.
A \textit{cut} is a half-space $\alpha^T x\geq \beta$ such that $\forall p \in P_I: \alpha^Tp\geq\beta$.
The cut \textit{yields} the hyperplane $\alpha^T x = \beta$.
A point $\Bar{x}$ is cut off by a cut if $\alpha^T \Bar{x} < \beta$.
For a clause $C=l_1 \lor ... \lor l_k$ we define the \textit{\clauseInequality{}} $\sum_{i=1}^k l_i \geq 1$.
As an example, the clause $x \lor \neg y$ gives the inequality $x+1-y \geq 1$.
If a \clauseInequality{} is a cut, we call it a \textit{\clauseCut{}}.
Finding a cut of a specific kind in order to add it to the problem description is called \textit{separating} a cut.
For example, finding a \clauseCut{} is called separating a \clauseCut{}.

Viewed as linear inequalities, \clauseCut{}s coincide with \textit{\noGoodCut{}s} as defined in \cite{bockmayr_detecting_2006}.
However, we introduce the term \clauseCut{} to highlight that these inequalities arise from clauses implied by the CNF encoding of the feasible set.
In this paper, we focus on identifying strong \clauseCut{}s that can be added to the MILP in order to improve the tightness of its \lpRelaxation{}, thus hoping to speed up the optimization process akin to a branch-and-cut procedure.
In our computational study however, we only cut at the root node of the branch-and-bound tree.

In this work we deal with weighted partial \maxsatProblem{}s in the following form, which we call \cnfoptProblem{}s:
\begin{equation}
    \underset{x}{\min}\;\; c^Tx \;\; \text{s.t.} \;\; x ~\text{satisfies}~\varphi, \;\; \forall x_i \in \text{var}(\varphi): x_i \in \mathbb{B} ,
\end{equation}
with an integral vector $c$ and the CNF $\varphi$.
This is basically an ILP that is constrained such that the feasible solutions are exactly the assignments that satisfy the CNF $\varphi$.
For any such ILP, the following proposition holds:
\begin{proposition}[Implied clauses yield \clauseCut{}s]
  For any clause $C: \varphi \implies C$, $C$'s \clauseInequality{} is a \clauseCut{} for the ILP constrained by $\varphi$.
\end{proposition}
\begin{proof}
    Suppose that for the clause $C$ it holds that $\varphi \implies C$ and that $x$ is a feasible solution to the ILP.
    As $x$ is a feasible solution to the ILP, it must satisfy $\varphi$.
    As $\varphi \implies C$, $x$ must satisfy $C$.
    Hence, $x$ must satisfy $C$'s \clauseInequality{}.
    As the argument holds for any feasible $x$, it holds for all of them.
    Because of this and as $P_I$ is convex (by its definition), the \clauseInequality{} of the implied clause $C$ is indeed a \clauseCut{}.
\end{proof}
\noindent A general weighted partial \maxsatProblem{} can be brought into \cnfopt{} form in the following way.
Let $\varphi$ be composed out of the hard clauses of the \maxsatProblem{} and then, for each soft clause $C$, an additional hard clause $C':=(C \lor x_C)$ with an auxiliary variable $x_C$ is added to $\varphi$.
The objective coefficient of $x_C$ is set to the weight of $C$ ~\cite{ansotegui_solving_2013}.
Conversely, to translate a \cnfoptProblem{} into a (weighted partial) \maxsatProblem{}, the clauses in the CNF $\varphi$ become the hard clauses.
For the terms in the objective, it depends on the sign of the objective coefficient: 
if $c_i\geq 0$, the soft clause $(\neg x_i)$ with weight $c_i$ is added.
If, on the other hand, $c_i < 0$, then we can rewrite $c_ix_i$ as follows:
\begin{equation}
    c_ix_i=-c_i(1-x_i)+c_i=-c_i \neg x_i + c_i .
\end{equation} 
We then add the soft clause $(x_i)$ with the positive weight $-c_i$ instead. 
The constant term $+c_i$ resulting from the rewriting can be ignored in the objective function, as it does not affect where the optimum is attained. \\

\noindent We also note that \clauseInequalities{} are generally not enough to describe the integral polytope $P_I$ of a \cnfoptProblem{}.
As an example we mention the stable-set-problem of a fully connected conflict graph of $n \geq 2$ nodes, where, as we shall see, none of the facets are \clauseInequalities{}.

\begin{proposition}[not all facets are \clauseCut{}s]
    \label{proposition:notFacets}
    Consider the maximum-stable-set problem of a fully connected conflict graph with $n\geq 2$ nodes. 
    Then there is no \clauseCut{} that can cut off the solution to its \lpRelaxation{}.
\end{proposition}
\begin{proof}
    For this problem, the feasible solutions are the corners of the unit simplex and thus $P_I=\{ x \in \mathbb{R}^n \mid x\geq0, x_1+...+x_n \leq 1 \}$.
    Written as \cnfopt{}, the problem takes the form 
    \begin{equation}
        \underset{x \in \mathbb{B}^n}{\min}\;\; x_1 + ... + x_n \;\; \text{s.t.} \;\; \forall i \neq j: ~ (\neg x_i \lor \neg x_j) ,
    \end{equation}
    or, written with linear inequalities,
    \begin{equation}
        \underset{x \in \mathbb{B}^n}{\min}\;\; x_1 + ... + x_n \;\; \text{s.t.} \;\; \forall i \neq j: ~ x_i + x_j \leq 1 .
    \end{equation}
    The solution to the problem's \lpRelaxation{} is $\Bar{x}_1=...=\Bar{x}_n=\frac{1}{2}$.
    First, there is no clause of length 1 that is implied by $\varphi$, as such a clause would fix a variable and thus be invalid.
    Suppose now that there was any implied clause $C=l_1 \lor ... \lor l_k $ of length $k\geq2$.
    As $\forall i: \Bar{x}_i=\frac{1}{2}$, it follows that for each literal $\Bar{l}$ it holds that $\Bar{l}=\frac{1}{2}$, and thus $\Bar{l}_1+...+\Bar{l}_k=\frac{1}{2}+...+\frac{1}{2} = 1 +...\geq 1$.
    Therefore, $\Bar{x}$ is not cut off by $C$'s \clauseInequality{}.
\end{proof}

\begin{remark}
    The missing facet $x_1+...+x_n \leq 1$ (a clique inequality) cannot be verified with a \satSolver{} directly, as it is not a \clauseInequality{}. 
\end{remark}

\subsection{SAT-solvers}
\satSolver{}s \cite{biere_handbook_2021} take a given CNF $\varphi$ and can then check whether the CNF is satisfiable (\textit{\sat}) or unsatisfiable (\textit{\unsat}).
After a solver was initialized with a CNF $\varphi$, in our notation, a call to the solver checking the satisfiability of $\varphi$ is made with $solve()$.
We want to emphasize that \satSolver{}s are not abstract functions but programs with special functionality that we use, and calls to the solver might change the solver's internal state.

Often, it is of interest whether a (partial) assignment can be extended to a satisfying assignment.
Such a (partial) assignment of literals, which is given to the solver as an optional argument to the method $solve()$ and which is assumed to be true only in this specific call, is commonly called \textit{assumptions}.
With the assumptions given, the solver then tries to find an extension of the partial assignment to a complete assignment that satisfies $\varphi$. 
If such an extension exists, it returns \sat, else it returns \unsat.
Assumptions can also be understood as temporarily adding their literals as unit clauses to the CNF only for this call to the solver. 
As an example in our notation, given a solver initialized with the CNF $\varphi=(x \lor y)\land (\neg y)$, $solve()$ would return \sat{} while $solve(\neg x)$ would return \unsat.
We call assumptions $a$ satisfiable exactly if they can be extended to an assignment that satisfies $\varphi$.
Unsatisfiable assumptions are also called an \textit{unsatisfiable core}. 
An unsatisfiable core is a \textit{minimal unsatisfiable core} if removing any one of its literals makes it satisfiable.
In the example, $\neg x$ is a minimal unsatisfiable core.
As a core is a conjunction of literals, the core can be ruled out by a clause containing the negated literals of the core, for example the unsatisfiable core $\neg x$ is ruled out with the clause $(x)$. 

We use \satSolver{}s in this work because of their ability to efficiently check satisfiability under assumptions, and the solver's consequent ability to check if a given clause $C$ is implied by the CNF $\varphi$, as this simultaneously proves that the \clauseInequality{} of $C$ is a \clauseCut{}.
To check whether $\varphi \implies C$ using assumptions, first note that $C$ is violated by an assignment exactly if all of $C$'s literals are violated.
Hence, an assignment that satisfies $\varphi$ but violates $C$ must satisfy $\varphi \land \bigwedge_{l \in C} \neg l$.
If $\varphi \land \bigwedge_{l \in C} \neg l$ is \unsat{}, no such assignment exists and thus all assignments satisfying $\varphi$ must also satisfy $C$, which means that $\varphi \implies C $.
Thus, if the \satSolver{} is queried with $solve(\land_{l \in C}\neg l)$ and returns \unsat{}, it holds that $\varphi \implies C$.
In the above example of $\varphi=(x \lor y)\land (\neg y)$, $solve(\neg x) =\unsat$ proves that $\varphi \implies (x)$.
As the clause $(x)$ corresponds to the \clauseInequality{} $x\geq 1$, this also proves that $x \geq 1$ is a cut for a corresponding \cnfoptProblem{}.

\textit{CDCL-based solvers} attempt a depth-first search for a satisfying assignment. 
If a conflict in the current (partial) assignment is detected, it is analyzed and a new (logically redundant) clause, which rules out this conflict, is learned from it.
An overview of the workings of CDCL \satSolver{}s can be found in \cite{biere_handbook_2021}.
CDCL-based solvers allow for warm-starts, where clauses learned from previous \satCall{}s can be reused in subsequent calls to the solver.
The clauses are learned such that they are implied by $\varphi$ even if they are derived in a call using assumptions.
In this paper, we use the ability of CDCL-based \satSolver{}s to learn clauses implied by the CNF, and we retrieve the set of all such clauses learned by the solver until now with $getLearnedClauses()$.

\subsection{Minimal unsatisfiable cores make for supporting hyperplanes}
\label{sec:supportingHyperplanes}
As shown in Proposition~\ref{proposition:notFacets}, \clauseCut{}s are not facet-defining in general.
Nevertheless, we prove in this section that any \clauseCut{} derived from a minimal unsatisfiable core defines a supporting hyperplane of the integer feasible polytope. 
Hence \clauseCut{}s can be strengthened to yield supporting hyperplanes;  
this guarantees a baseline quality for the cuts produced by our algorithms in Section \ref{sec:algorithms}, in the sense that the right-hand-side of the constraint cannot be strengthened without making the cut invalid.

\begin{proposition}[Necessary literals in unsatisfiable cores make for supporting hyperplanes]
\label{proposition:supportingHyperplanes}
    Let $C=l_1 \land ... \land l_k$ be an unsatisfiable core with the literal $l_1$ such that $D=l_2 \land ... \land l_k$ is satisfiable.
    Then the \clauseCut{} of the clause $ \neg l_1 \lor ... \lor \neg l_k$ yields a supporting hyperplane of the polyhedron.
\end{proposition}
\begin{proof}  
    By assumption it holds that     
    \begin{equation}
        \exists \text{ assignment } a : a(l_1)=0 \; \wedge (\forall j\neq 1: \, a(l_j)=1) \; \wedge \; (a \text{ satisfies } \varphi ) .
    \end{equation}
    That means for the literal $l_1$ in the unsatisfiable core $C$, there is an assignment $a$ that does not satisfy $l_1$ but every other literal in $C$ and $a$ also satisfies $\varphi$.
    This formalizes the demand that removing the literal $l_1$ from $C$ makes the core satisfiable.
    
    Consider now the assignment $a$.
    $a$ corresponds to the binary vector 
    \begin{equation}
        p \in \mathbb{B}^n: p_k=
        \begin{cases}
            0 & a(x_k)=0 \\
            1 & a(x_k)=1
        \end{cases} .
    \end{equation}
    $p$ is integrally feasible for the problem, as $a$ satisfies $\varphi$ by assumption.
    At the same time, $C$ yields the \clauseCut{}
    \begin{equation}
        \neg l_1 + ... + \neg l_k \geq 1 .
    \end{equation}
    Evaluating the vector $p$ in this inequality yields
    \begin{equation}
        a(\neg l_1)+\sum_{j\neq 1} a(\neg l_j)
        =1-a(l_1)+\sum_{j\neq 1} 1-a(l_j)
        =1+\sum_{j\neq 1} 0 = 1 .
    \end{equation}
    Hence $p$ is in the integer feasible polytope (as $p$ satisfies $\varphi$) and fulfills the \clauseCut{} with equality. 
    Thus the \clauseCut{} yields a supporting hyperplane of the integer feasible polytope.    
\end{proof}
\noindent With this we can conclude for minimal unsatisfiable cores: 
\begin{proposition}[Minimal unsatisfiable cores make for supporting hyperplanes]
    Let $C=l_1 \land ... \land l_k$ be a minimal unsatisfiable core with literals $l_i$.
    Then the \clauseCut{} of the clause $(\neg l_1 \lor ... \lor \neg l_k)$ yields a supporting hyperplane of the polyhedron.
\end{proposition}
\begin{proof}  
    As all literals in a minimal unsatisfiable core are necessary to keep the core unsatisfiable, we can pick any and apply Proposition \ref{proposition:supportingHyperplanes}.
\end{proof}

\section{Using \sat~Solvers to separate \clauseCut{}s}
\label{sec:algorithms}
In this chapter we introduce two simple algorithms that try to separate \clauseCut{}s.
For this, the algorithms inspect a given point $\Bar{x}$ and then try to find a valid \clauseCut{} that is violated by $\Bar{x}$.
In an \milp{}-context, $\Bar{x}$ would be the (fractional) solution to the current relaxation, and adding the \clauseCut{} would strengthen this relaxation with the goal of approximating the integral hull $P_I$ as well as possible.

As the algorithms need to solve a general \satProblem{}, they might take substantial time to complete.
However, the use of modern \satSolver{}s enables efficient derivation and validation of \clauseCut{}s in practice. 
Furthermore, the solvers can also be used to strengthen the \clauseCut{} by shortening the clause to which the cut corresponds. 
Neither of the algorithms presented is complete, meaning there are instances where a separating \clauseCut{} for $\Bar{x}$ exists but is not found.
Algorithm \ref{alg:proof} can, by design, separate a superset of \clauseCut{}s compared to Algorithm \ref{alg:basic}. 
If the entries of $\Bar{x}$ with integral values form unsatisfiable assumptions, both algorithms are guaranteed to find a \clauseCut{}.

\subsection{A motivating example for \clauseCut{}s}
\label{sec:algorithms:example}
In general, \cnfopt{} formulations can allow for overly optimistic \lpRelaxation{}s which circumvent the logic actually encoded in the CNF.
For example, setting all variables to the fractional value of $0.5$ will satisfy all inequalities associated to clauses of length 2 or greater.
This can be especially disadvantageous if the target variables (objective coefficients $\neq 0$) of the optimization are determined by a system of logical implications from logically interconnected non-target (objective coefficient $=0$) variables.
We will now examine an example that illustrates the difficulties of \cnfopt{} as well as the utility of our proposed cuts. \\
Consider the problem
\begin{mini}
    {x,y,z \in \mathbb{B}}{z}{}{}
    \addConstraint{x }{\iff y}
    \addConstraint{z}{\iff (x = y)}.
\end{mini}
We note that as $x$ is equivalent to $y$ in this problem, it must always hold that $z=1$.
We could also present the problem as the following CNF:
\begin{mini}
    {x,y,z \in \mathbb{B}}{z}{}{}
    \addConstraint{\neg x }{\lor \;\; y}{}
    \addConstraint{x }{\lor \neg y}{}
    \addConstraint{\neg x }{\lor \neg y }{\lor \;\; z}
    \addConstraint{\neg x }{\lor \;\; y }{\lor \neg z}
    \addConstraint{x }{\lor \neg y }{\lor \neg z}
    \addConstraint{x }{\lor \;\; y }{\lor \;\; z}.
\end{mini}
This corresponds to the MILP 
\begin{mini}
    {x,y,z \in \mathbb{B}}{z}{}{}
    \addConstraint{1-x}{ + \;\;\;\;\;\; y}{\;\;\;\;\;\;\;\;\; \;\;\; \geq 1 }
    \addConstraint{ x   }{ +1-y}{\;\;\;\;\;\;\;\;\; \;\;\; \geq 1 }
    \addConstraint{1-x}{ +1-y}{+ \;\;\;\;\;\; z \geq 1 }
    \addConstraint{1-x}{ + \;\;\;\;\;\; y   }{+ 1-z \geq 1 }
    \addConstraint{ x   }{ +1-y}{+ 1-z \geq 1 }
    \addConstraint{ x   }{ + \;\;\;\;\;\; y }{+ \;\;\;\;\;\; z \geq 1 }.
\end{mini}
The unique solution of this problem's \lpRelaxation{} is 
\begin{equation}
    x=0.5, \; y=0.5, \; z=0 .
\end{equation}
An inspection of the \lpRelaxation{} reveals that the variable $z$ takes the value zero -- contradicting our earlier observation that $z = 1$ is logically implied by the original CNF. 
If the implied information can be expressed as a clause, it follows from Section \ref{sec:Basics} that it can be translated into a \clauseCut{} and directly added to the MILP as a valid constraint.  
In the current example, the inequality $z \geq 1$ corresponds to the final missing facet of the integer hull.
Once incorporated, the \lpRelaxation{} becomes tight enough to yield the exact integer optimum, effectively solving the MILP through its relaxation.
This motivates a central question: \textit{How can we systematically derive implied clauses from the original CNF that strengthen the MILP by tightening its relaxation as much as possible?}

A practical approach, which, to the best of our knowledge, has not yet been used in the literature, is to examine the integral variables of the LP solution and to check whether their assignment violates the CNF. 
In other words, we test whether this partial assignment is already unsatisfiable. 
If so, it constitutes a \noGoodCut{} -- a clause that excludes an infeasible assignment of binary variables.
This test is implemented by passing the integral variables of the \lpRelaxation{}'s solution as assumptions to a SAT solver. 
If the solver returns \unsat{}, the assignment contradicts the CNF and thus implies a valid clause that rules out the current LP-solution. 
In our case, the LP solution assigns $z = 0$ as the only integral variable. 
Passing this as an assumption to a solver initialized for $\varphi$ yields:
\begin{equation*}
    solve(\neg z) = \unsat{} .
\end{equation*}
From this, we recover the implied clause $(z)$ and, correspondingly, the \clauseCut{} $z \geq 1$.

Importantly, adding this inequality completes the description of the integer feasible polytope. 
The updated \lpRelaxation{} now yields the integer-feasible solution $(x, y, z) = (1, 1, 1)$, solving the problem without further branching.

\subsection{The \baseMethod{}}
\label{sec:algorithms:base}
The first method directly implements the idea from Section \ref{sec:algorithms:example}.
For this method, we propose to use a SAT oracle to check whether the binary variables that are integral at the \lpRelaxation{}'s solution of the ILP form an unsatisfiable core, see Section \ref{sec:Basics} for definitions.
For this purpose, any \satSolver{} can be used.

If such an unsatisfiable core exists, it can be ruled out by adding a clause, thereby yielding a \clauseCut{}. 
This \clauseCut{} can often be strengthened by iteratively removing literals from the unsatisfiable core, yielding a minimal unsatisfiable core. 
Removing literals from the clause corresponds to eliminating non-negative terms from the left-hand side of the resulting cut, while leaving its right-hand side unchanged. 
As a consequence, the cut becomes stronger, being more restrictive with respect to the terms it still contains.
In Section \ref{sec:supportingHyperplanes} we showed that the \clauseCut{} derived from such a minimal core yields a supporting hyperplane of the integer-feasible polytope.

The \baseMethod{} (\baseMethodShort{}) described here is guaranteed to find a \clauseCut{} if and only if the partial assignment given by the integral variables of the LP solution is an unsatisfiable core.
Hence, if all valid \clauseCut{}s necessarily involve at least one currently non-integral variable, the method will not detect them.

\vspace{\baselineskip}
\noindent \textbf{Explanation of the \baseMethod{} (\baseMethodShort{}):}
The algorithm requires a \satSolver{} initialized with the CNF $\varphi$ and takes as input the fractional point $\Bar{x}$.
In line 1 we save the partial assignment given by the integral variables in $\Bar{x}$ as assumptions in the variable $A$.
In line 2, the solver is then queried with $solve(A)$ to check if these assumptions can be extended to a feasible solution.
If not, the solver returns \unsat{} and $A$ is an unsatisfiable core that represents the valid \noGoodCut{} $\sum_{i: \Bar{x}_i=0} x_i + \sum_{i: \Bar{x}_i=1} 1-x_i \geq 1$ which cuts off $\Bar{x}_i$.
If the solver returns SAT, no cut can be found by the \baseMethodShort{} and the algorithm returns $None$ in line 10.
In case a \clauseCut{} is found, it is consequently strengthened through a knock-out procedure to yield a supporting hyperplane from line 3 onward through making the unsatisfiable core into a minimal unsatisfiable core.
For the knock-out procedure we iterate over each literal in the core and check in line 4 whether leaving it out leaves the remaining core unsatisfiable. 
If this is the case, the literal is actually removed from the core and we proceed to the next candidate literal.
After strengthening, the cut is returned in line 8.

These cuts only operate on the set of literals given by the integral variables of the LP solution. 
This limits the set of possible cuts, as there are examples where the \lpRelaxation{} will not have any integral components (see Section \ref{sec:exampleNoIntegral}).
However, as we will see in Section \ref{sec:study}, the cuts can be very useful in practice.
The algorithm performs a single \satCall{} to check the integral values for feasibility, thereby finding some clause for a cut or verifying that it cannot find any. 
In practice there are many integral variables in the LP-solution, usually making this call quite fast, as many variables are fixed through the assumptions in the \satCall{}.
(Note that possible branched variables are a subset of those integral variables, and the cuts hence may prune the branch and bound tree.)
Additionally, it requires as many \satCall{}s for strengthening as the initial clause is long.
Though each of these \satCall{}s could take long due to the NP-complete nature of the problem, they are very fast in practice, as the \satSolver{} -- usually a CDCL solver -- warm starts by heavily reusing clauses it learned during the previous knock-outs. 

\vspace{\baselineskip}
\begin{algorithm}[H]
\label{alg:basic}
    \caption{\baseMethod{} (\baseMethodShort{})}
    \SetAlgoLined
    \LinesNumbered
    \DontPrintSemicolon
    \SetKwInOut{Input}{Input}
    \SetKwInOut{Output}{Output}
    \Input{vector $\Bar{x} \in \mathbb{R}^n$}
	\Output{\clauseCut{} cutting off $\Bar{x}$ or None}
    $A \gets \text{partial assignment}: \{x_i \mid \Bar{x}_i \in \mathbb{B}\} \rightarrow \mathbb{B}, \; x_i \mapsto \Bar{x}_i$ \;
	\eIf{$solve(A)=\unsat$}{  
         \For{$literal \in A$}{     
            \If{$solve(A \setminus \{lit\}) = \unsat$}{
                $A \gets A \setminus \{lit\}$ }}
        \Return $\sum_{l_i \in A}1-l_i \geq 1$}
	{
        \Return $None$    }
\end{algorithm}

\subsection{The \proofMethod{}}
\label{sec:algorithms:proof}
As discussed in Section \ref{sec:Basics}, every valid clause derivable from the CNF constitutes a valid \clauseCut{}.
Hence all clauses learned by a CDCL \satSolver{} (see Section \ref{sec:Basics}) yield cuts that can potentially be utilized within the MILP framework for cutting off fractional solutions.
Consequently, this algorithm assumes the use of a CDCL-based \satSolver{} (or any \satSolver{} capable of learning implied clauses), rather than relying merely on a generic \sat-oracle.
A notable advantage of this procedure is that it can separate every clause that has been derived by the \satSolver{} up to that point.

\vspace{\baselineskip}
\noindent \textbf{Explanation of the \proofMethod{} (\proofMethodShort{}):}
The algorithm requires a CDCL \satSolver{} initialized with the CNF $\varphi$ and takes as input the fractional point $\Bar{x}$.
This algorithm also uses an external data structure, \textit{knownClauses}, which is used as a database to save and look up clauses which are known to be valid for $\varphi$.
Similar to the \baseMethodShort{} in Section \ref{alg:basic}, the \proofMethodShort{} first builds in line 1 a partial assignment out of the integral components of the provided vector $\Bar{x}$. 
After that, in line 2, the solver is called, using the (possibly empty) partial assignment $A$ as assumptions.
Here, it is important to note that the \satSolver{} might learn clauses during any \satCall{}, even if the assumptions are satisfiable or if there are no assumptions at all.
Always calling the solver with assumptions dependent on $\Bar{x}$ also allows the solver to explore different parts of the space of possible clauses. 
From line 3 onward, the algorithms differ however: in case the solver returns \unsat{}, we add the corresponding clause to the database $knownClauses$ containing the known valid clauses. 
Thereby the algorithm is guaranteed to find a cut if the partial assignment is \unsat{}, similar to the \baseMethodShort{}.
Adding the clause ``manually'' is necessary, because the solver itself might not explicitly learn a clause that rules out the unsatisfiable assumptions, for example, if the assumptions lead to \unsat{} via unit propagation. 

In line 5 we call $getLearnedClauses()$ to read the new clauses the solver has learned up to this point -- if there are any --  and add them to the database \textit{knownClauses}.
Following this, in line 6, the method $checkKnownClausesForViolation$ with the argument $\Bar{x}$ searches the \textit{knownClauses} for any known clause whose \clauseCut{} is violated by $\Bar{x}$. 
There might be several such clauses.
However, as \clauseCut{}s usually must be strengthened to be useful, we do not try to strengthen and add all of them, but instead use a tie braking rule to select one with which we proceed.
In our experiments, we chose to use the first clause of shortest length among all violated clauses in $knownClauses$.
If a violated clause is found by $checkKnownClausesForViolation$, it is saved in the variable $C$, then strengthened in the same way as in the \baseMethodShort{}, and finally its corresponding \clauseCut{} is returned by the algorithm.
If no violated clause is found by $checkKnownClausesForViolation$, the algorithm returns $None$.
In terms of \satSolver{} usage, the \proofMethodShort{} is identical to the \baseMethodShort{}.
However, the \proofMethodShort{} has the additional step of searching through the known clauses, which allows it to possibly find cuts the \baseMethodShort{} can not. 
As the \proofMethodShort{} is also guaranteed to find a cut if the integral values of the supplied fractional point form an unsatisfiable core, the cuts it can separate form a superset of the \baseMethodShort{}'s cuts.

\vspace{\baselineskip}
\begin{algorithm}[H]
\label{alg:proof}
    \caption{\proofMethod{} (\proofMethodShort{})}
    \SetAlgoLined
    \LinesNumbered
    \DontPrintSemicolon
    \SetKwInOut{Input}{Input}
    \SetKwInOut{Output}{Output}
    \Input{vector $\Bar{x} \in \mathbb{R}^n$}
	\Output{\clauseCut{} cutting off $\Bar{x}$ or $None$}
	$A \gets \text{partial assignment}: \{x_i \mid \Bar{x}_i \in \mathbb{B}\} \rightarrow \mathbb{B}, \; x_i \mapsto \Bar{x}_i$ \;
	\If{$solve(A)=\unsat$}{
            add clause $\bigvee_{l_i \in A} \neg l_i $ to the $knownClauses$ 
    }
    $knownClauses \gets knownClauses \cup getLearnedClauses()$  \;
    $C \gets checkKnownClausesForViolation(\Bar{x})$   \;
    \eIf{$C \neq None $}{
            $core \gets \text{partial assignment}: \bigwedge_{l_i \in C} \neg l_i $ \;
            \For{$lit \in core$}{
                \If {$solve(core \setminus \{lit\})=\unsat$}{
                    $core \gets core \setminus \{lit\}$ }}
            \Return $\sum_{l_i \in core}1-l_i \geq 1$}
    {\Return $None$   }
\end{algorithm}

\subsection{Limitations of \clauseCut{} separation algorithms based on integral values}
\label{sec:exampleNoIntegral}
Both the algorithms introduced in this section use the integral values of a supplied point to build assumptions which are passed to a \satSolver{}. 
This means that supplied points without any integral components might cause problems.
While a CDCL solver in Algorithm \ref{alg:proof} might still learn clauses while verifying the feasibility of the CNF $\varphi$, it is not guaranteed to do so; the Algorithm \ref{alg:basic} however is guaranteed to not yield any cut in this case.
We want to mention an example where Algorithm \ref{alg:basic} will indeed be unable to find a cut:
\begin{maxi} 
    {x,y \in \mathbb{B}}{x}{}{}
    \addConstraint{\neg x }{\lor \neg y}
    \addConstraint{\neg x}{\lor \;\; y}.
\end{maxi}
The two clauses $\neg x \lor \neg y$ and $\neg x \lor y$ imply the clause $\neg x$, which is violated by the unique solution of the problem's \lpRelaxation{} $(x,y)=(\frac{1}{2},\frac{1}{2})$.
However, the problem is feasible but neither $x$ nor $y$ are integral; hence Algorithm \ref{alg:basic} will not be able to produce any \clauseCut{}.
Algorithm \ref{alg:proof} might find the clause $\neg x$, whose \clauseCut{} would separate the LP-solution, but is not guaranteed to do so: 
if, during the CDCL \satSolver{}'s depth-first-search, the solver first branches to $x=0$, it will immediately find a satisfying assignment and not learn any clauses.
In this case, no \clauseCut{} will be found.
If, on the other hand, the solver first branches to $x=1$, it will immediately detect the infeasibility through unit propagation and learn the clause $(\neg x)$, thus discovering the desired cut. 

\section{Computational Study}
\label{sec:study}
In this study, we compare both of our algorithms in terms of runtime against the general purpose \milpSolver{} \gurobi{} 12 \cite{gurobi_optimization_llc_gurobi_2024} (Section \ref{sec:study:Gurobi}) as well as the specialized \maxsatSolver{} ``\rcTwo{}" \cite{ignatiev_rc2_2019} (Section \ref{sec:study:RC2}) from the Python package \pysat{} \cite{ignatiev_rc2_2019}.
For these comparisons, Algorithm \ref{alg:solveInstance} was used.
There, either one of our algorithms Algorithm \ref{alg:basic} (\baseMethodShort{}, Section \ref{sec:algorithms:base}) or Algorithm \ref{alg:proof} (\proofMethodShort{}, Section \ref{sec:algorithms:proof}) is used in a cutting plane method to strengthen the \lpRelaxation{} of the problem before the strengthened formulation is handed to the \milpSolver{}.
First, the feasibility of the problem is verified with the initial \satCall{} in line 2.
For the cutting plane procedure, the \lpRelaxation{}s are solved with \gurobi{}, the fractional solution analyzed with the algorithms from Section \ref{sec:algorithms}, and if a cut is found, it is added to the problem.
Then, the strengthened \lpRelaxation{} is solved again and the procedure repeated until a stopping criterion is met or no more cuts are found.
In this paper, total runtime always refers to the total time needed from reading the problem to the finished optimization.

Moreover, an interesting question is why our cuts lead to such substantial speedups.
The study thus investigates in Section \ref{sec:study:speedup} the relationship between the average length (equating to the strength) of the \clauseCut{}s and the speedup observed during optimization.
Following this, we analyze the relationship between the time taken by the initial \satCall{}, which verifies the feasibility of the problem, and the average clause lengths of the \clauseCut{}s found in Section \ref{sec:study:initialVsClauseLength}.
In Section \ref{sec:study:initialVsTotal}, we also compare the total runtime of the optimization using our algorithms to the time needed for the initial \satCall{}.
The time needed for the initial call provides a lower bound on the possible runtime of any optimization algorithm directly employing the \satSolver{}. 

Section \ref{sec:study:nodes} finally analyzes the impact of \clauseCut{}s generated by the \baseMethodShort{} and \proofMethodShort{} methods on the branch-and-bound trees during the optimization of the strengthened MILPs.
For this, we compare the number of nodes explored by \gurobi{} alone against those explored during optimization using our methods. 
During this analysis, we also examined whether the integral values taken in each node's LP-solution resulted in infeasibility. \\

\noindent All experiments in this study were conducted on \cnfoptProblem{}s constructed from the SATLIB uf250 ~\cite{hoos_satlib_nodate} dataset of randomly generated satisfiable 3\satProblem{}s.
It consists of 100 satisfiable 3-CNF formulae, each with 250 variables and 1065 clauses.
To generate a \cnfopt{} instance, the SATLIB instance's CNF was taken and a randomly generated integer objective function added.
The weights were sampled from the range $[-10, 10]$. 
The instances were encoded in wcnf format used in the MAXSAT Evaluations 2023~\cite{maxsat_evaluations_2023_site}.

The \satSolver{} used was Glucose 4.2~\cite{audemard_glucose_2018} from the package \pysat{}~\cite{ignatiev_pysat_2018}.
We have chosen Glucose 4.2 because it was the fastest solver in the \pysat{} library that we could find for our purposes.
To access the learned clauses of the \satSolver{} in the \proofMethodShort{}, we use the proof tracking provided by the solver as a means of accessing the learned clauses.
The implementation of \rcTwo{} we use is the one provided in \pysat{}.
 
All computations were conducted on machines supplied by the NHR@FAU High Performance Computing (HPC) cluster, using compute nodes with two Intel Xeon Gold 6326 ("Ice Lake") processors (32 cores, 2.9 GHz), and 256 GB of main memory.
The operating system was AlmaLinux 8.10 (Cerulean Leopard, 64-bit) running Python 3.12.8. 
Optimization was performed with Gurobi Optimizer version 12.0.0 (build v12.0.0rc1) via the gurobipy (version 12.0.0) interface.
The \pysat{} version was 1.8.dev13.

To solve the instances with our algorithms, the following procedure was applied: \\
\begin{algorithm}[H]
    \label{alg:solveInstance}
    \caption{Solve Instance}
    \SetAlgoLined
    \LinesNumbered
    \DontPrintSemicolon
    \SetKwInOut{Input}{Input}
    \SetKwInOut{Output}{Output}
    \SetKwRepeat{Do}{do}{while}
    
    \Input{\cnfopt{} instance $p$ with CNF $\varphi$, linear objective function $c$, $separationAlgorithm$ is either ICCA or LCCA}
	\Output{optimal solution of $p$}    
    $\text{create \satSolver{} initialized with } \varphi$ \;
    $solve()$ \;
    $nCuts \gets 0$ \;
    $mip \gets \text{MILP model of } p$ with objective $c$ \;
    \Do {$cut \neq None \land nCuts < 150 \land elapsedTime <60\;seconds$}{ 
        optimize \lpRelaxation{} of $mip$ with \gurobi{} \;
        $\Bar{x} \gets $ optimal solution to \lpRelaxation{} of $mip$ \;
        $cut \gets separationAlgorithm(\Bar{x})$ \;
        \If {$cut \neq None$}{
            add $cut$ to $mip$ \;
            $nCuts \gets nCuts+1$}}
    solve $mip$ with \gurobi{} \;
    \Return optimal solution to $mip$
\end{algorithm}

\subsection{Runtime Comparison against \gurobi{} 12}
\label{sec:study:Gurobi}
In Figure \ref{fig:runtimesVsGrb} we compare the total runtimes of the general purpose \milpSolver{} \gurobi{} 12 and our methods.
Both of our methods outperform the general purpose \milpSolver{} \gurobi{} significantly, with a larger lead for the \proofMethodShort{}.
The graph shows each problem instance in the benchmark dataset as a data point in the plot.
A data point's $x$-axis value is given by its runtime in \gurobi{}, while its $y$-axis value is its runtime in our approach. 
To visualize the vastly different runtimes, the graph is logarithmic in both axes.
The graph contains two data sets: the runtimes for the \baseMethodShort{} as well as the runtimes of the \proofMethodShort{}. 
For reference, we added a line of equal runtimes. 
That means, if a point was on the line, its instance's runtime in \gurobi{} would be equal to the runtime in our method. 
Consequently, data points below the line are instances which have been solved faster by our methods than by \gurobi{}. 
Almost all instances are solved faster by the \baseMethodShort{} than by \gurobi{}, and the \proofMethodShort{} outperforms \gurobi{} in all instances. 
Both methods achieve very high speedups, which reach up to two orders of magnitude compared to \gurobi{}.
Employing the \proofMethodShort{} allows for solving the entire benchmark set in only 7.8\% of the time needed by \gurobi{} alone.
Finally, we note that while the problems seem to be almost uniformly difficult for \gurobi{}, two clusters of varying runtime seem to emerge for our algorithms, with the more difficult cluster being significantly smaller when using the \proofMethodShort{}.

\begin{figure}
    \includegraphics[width=0.6\textwidth]{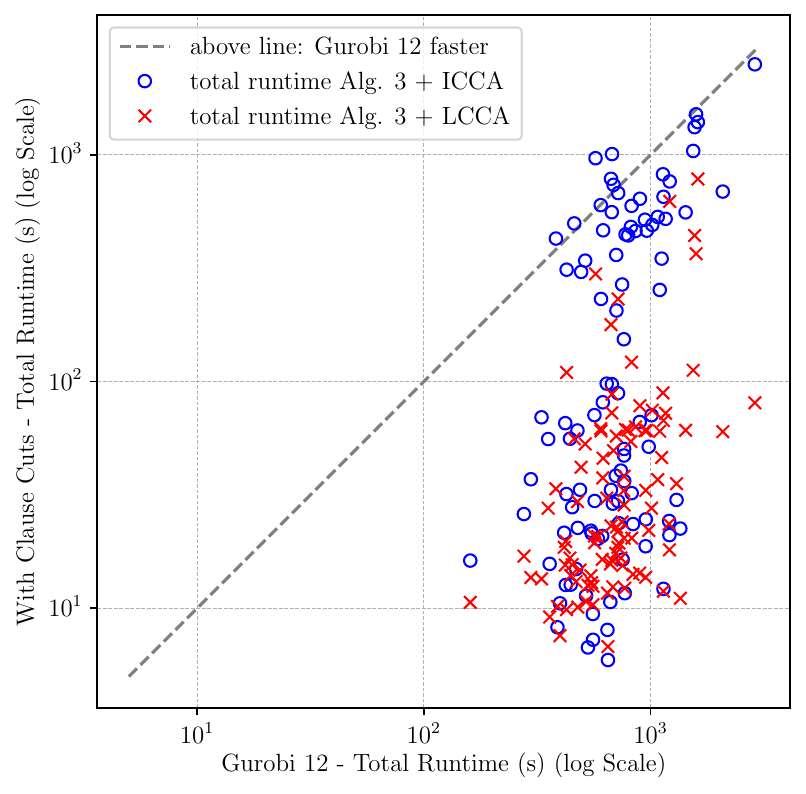}
    \centering
    \caption{Total runtimes of \gurobi{} 12 vs. Alg. \ref{alg:solveInstance} with \baseMethodShort{} and \proofMethodShort{}.}
    \label{fig:runtimesVsGrb}
\end{figure}

\subsection{Runtime comparison against the \maxsatSolver{} \rcTwo{}}
\label{sec:study:RC2}
For a comparison of our algorithms against a specialized \maxsatSolver{}, we chose \rcTwo{} included in the \pysat{}-package. 
In the MAXSAT Evaluations 2018 \cite{maxsat_2018_weighted} and 2019 \cite{maxsat_2019_weighted}, \rcTwo{} was ranked first in two complete categories, unweighted and weighted.
(In the complete category, the algorithms must produce a global optimum.)
The comparison against our algorithms is presented in Figure \ref{fig:runtimesVsRC2}.
For explanations on the make up of Figure \ref{fig:runtimesVsRC2} we refer to Section \ref{sec:study:Gurobi}, as it is identical to Figure \ref{fig:runtimesVsGrb}.

\rcTwo{}, which is a specialized solver for \maxsatProblem s, can solve the instances much faster than the general purpose \milpSolver{} \gurobi{}.
While the \baseMethodShort{} can be the fastest method on easy instances, it is falling behind on the harder instances comprised of the more difficult instance cluster and outperforms \rcTwo{} only in 35 out of 100 instances.
The \proofMethodShort{} however can solve 88 out of 100 instances faster than \rcTwo{} and solves the entire problem set in only 60.1\% of the time needed by \rcTwo{}.

\begin{figure}
    \includegraphics[width=0.6\textwidth]{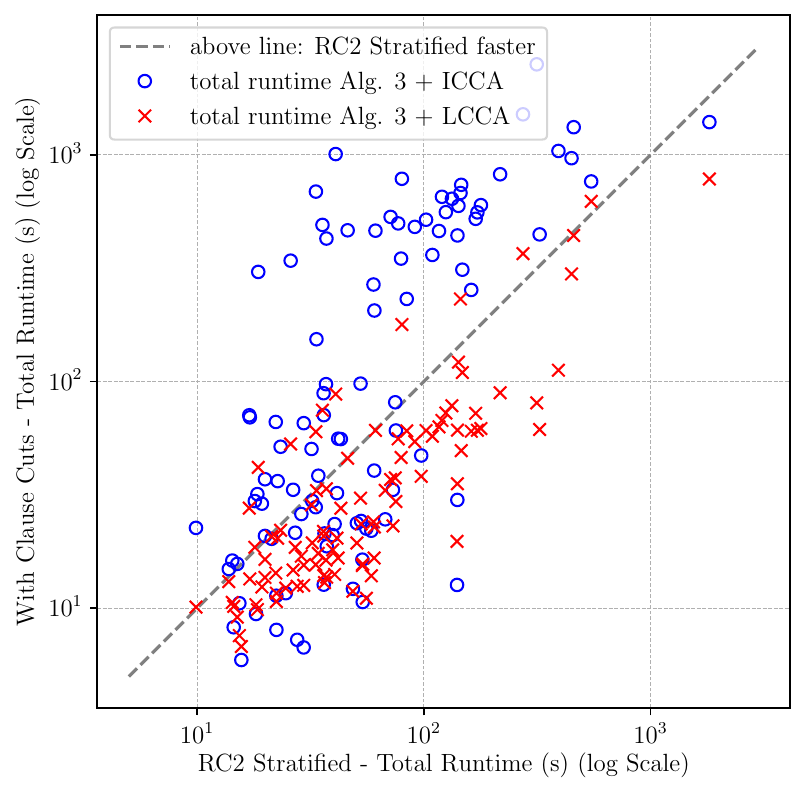}
    \centering
    \caption{Total runtimes of \rcTwo{} vs. Alg. \ref{alg:solveInstance} with \baseMethodShort{} and \proofMethodShort{}.}
    \label{fig:runtimesVsRC2}
\end{figure}

In contrast to the MILP based approaches, \rcTwo{} seemed to display sensitivity towards the magnitude of the objective functions during the computational study, getting slower as objective coefficients increased.
We provide cactus plots for both smaller and larger objective ranges in Section \ref{sec:study:Cactus}.

\subsection{Overall Comparison}
\label{sec:study:Cactus}
To compare all the different algorithms together and illustrate their performance over the whole problem set, we provide a cactus plot of the results in Figure \ref{fig:cactus10}. 
In their logarithmic $y$-axis, the graphs show the cumulative runtime of the $n$ fastest instances for each algorithm, with $n$ being the value on the $x$-axis.
As an example, choosing $n=50$ yields the times needed by the different solvers to solve their faster half of the benchmark set.
As the $n$ fastest instances are taken on a solver to solver basis, a cactus plot does not allow comparing two algorithms on the same instance.
For larger objective ranges, \rcTwo{} runs into time out at 7200 seconds. 
In the graphs, the $y$-axis value to represent time out is chosen to be $5 \cdot 10^5$ seconds.

\begin{figure}
    \centering
    \includegraphics[width=0.80\textwidth]{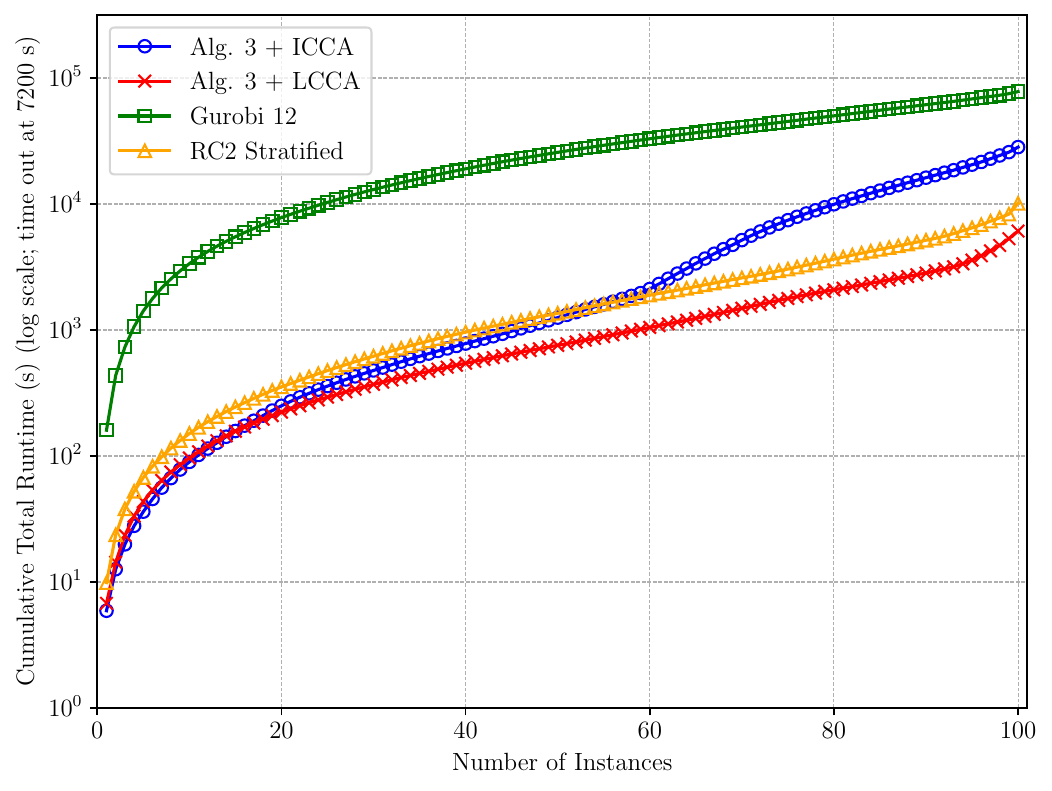}
    \caption{Cactus plot for objective range [-10,10] comparing Alg. \ref{alg:solveInstance} with \baseMethodShort{}, \proofMethodShort{}, \gurobi{} 12 and \rcTwo{}.}
    \label{fig:cactus10}
\end{figure}

To make our study robust against possible influences of different objective ranges, we also conducted a run over the full data set with objective coefficients in the range $[-1,1]$ (Figure \ref{fig:cactus1}), which is the best case for \rcTwo{} assuming its sensitivity towards larger objective coefficients.
Even in this case, the \proofMethodShort{} outperforms \rcTwo{}.
For integer objective coefficients sampled in the range of $[-25,25]$, \rcTwo{} is often running into time-out (see Figure \ref{fig:cactus25}).
The MILP based approaches \gurobi{} 12 only, \baseMethodShort{} and \proofMethodShort{} do not display sensitivity towards the magnitude of the objective coefficients.

\begin{figure}
    \centering
    \begin{minipage}{0.5\textwidth}
        \includegraphics[width=0.90\textwidth]{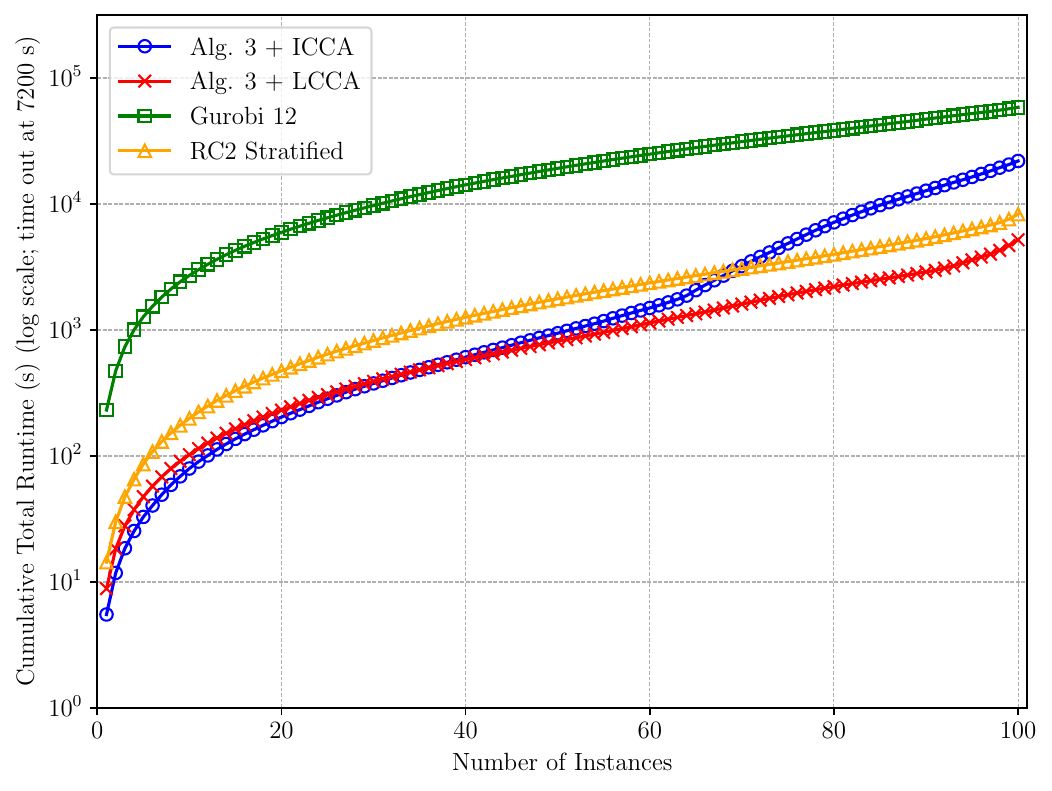}
        \caption{Cactus Plot for objective range $[-1, 1]$.}
        \label{fig:cactus1}
    \end{minipage}\hfill
    \begin{minipage}{0.5\textwidth}
        \includegraphics[width=0.90\textwidth]{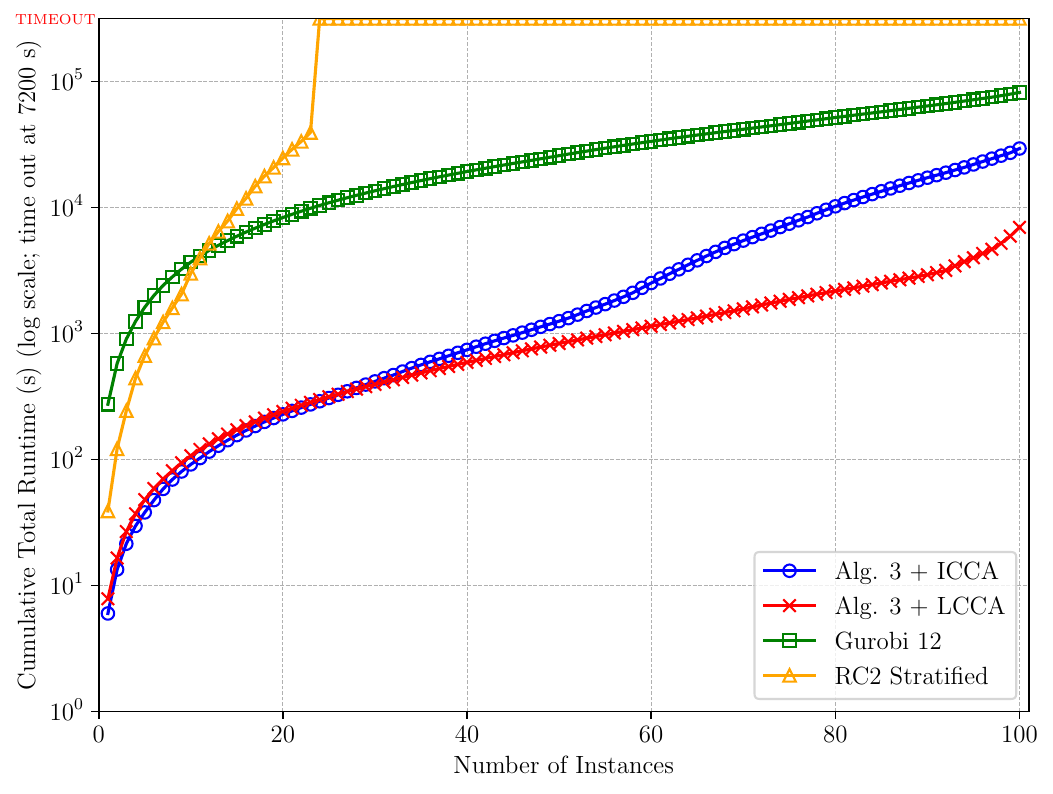}
        \caption{Cactus Plot for objective range $[-25, 25]$.} 
        \label{fig:cactus25}
    \end{minipage}
\end{figure}

\subsection{Connection between the Average Length of the \clauseCut{}s on the Total Speedup}
\label{sec:study:speedup}
An interesting question is where the speedups from our methods come from.
To inquire this, Figure \ref{fig:ClauseLengthVsSpeedup} shows the impact of the strength of the \clauseCut{}s, measured by the length of their clauses, on the speedup.
The study shows that the average clause length ($x$-axis) and thus the strength of the cuts has indeed a strong impact on the achieved speedup ($y$-axis, log scale), showing a nearly exponential dependence of the speedup on the average clause length.
This speedup becomes negligible for average clause lengths of 4 or greater, where they cannot achieve any significant speedup. 
Thus, for solving the benchmark instances, it seems essential to find strong \clauseCut{}s of length one (fixing a variable) or two (a conflict between two variables). 
Interestingly, many instances produced \clauseCut{}s with an average length of or very close to 1, indicating that many of these random problems have implicitly fixed variables.
Also, the strengthening step seems to be important to make the \clauseCut{}s useful.
We furthermore note that the \proofMethodShort{} yields significantly stronger cuts, as is reflected in its shorter average clause lengths. 
\begin{figure}
    \includegraphics[width=0.66\textwidth]{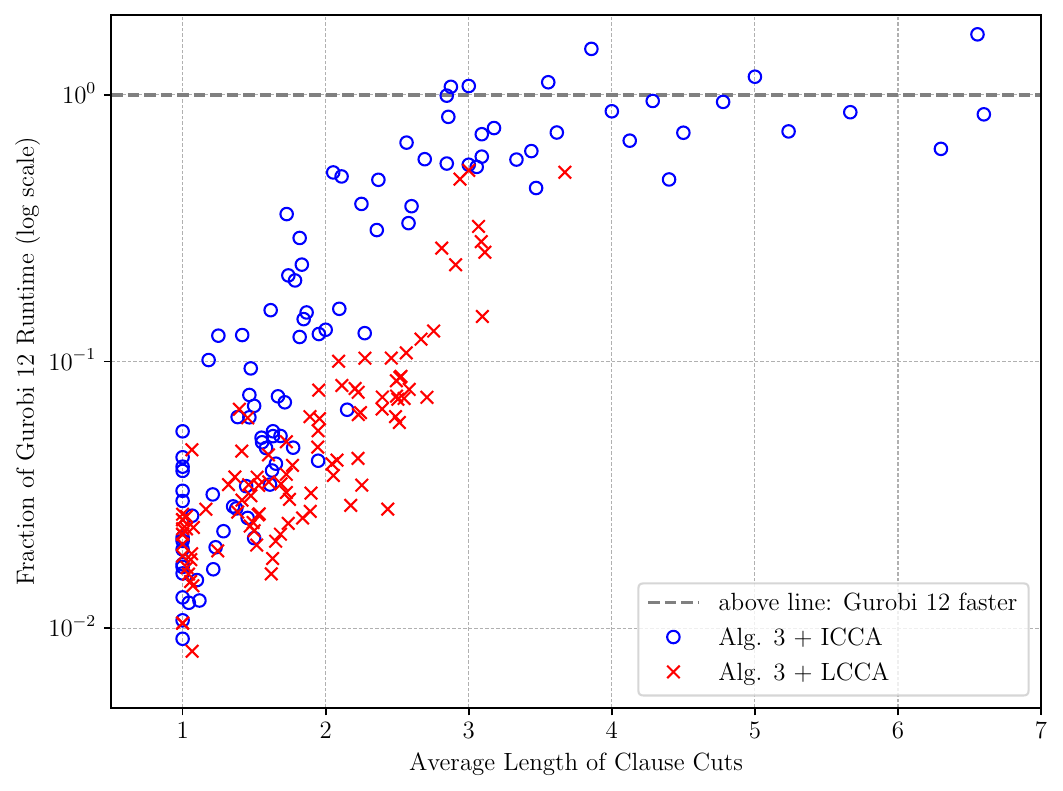}
    \centering
    \caption{Plot relating the average length of the \clauseCut{}s to the ratio $\frac{runtime_{\gurobi{}}}{runtime_{Alg.\ref{alg:solveInstance}}}$.
    }
    \label{fig:ClauseLengthVsSpeedup}
\end{figure}

\subsection{Relationship between the Time needed for the Initial \satCall{} and the Average Length of the \clauseCut{}s}
\label{sec:study:initialVsClauseLength}
The graph in Figure \ref{fig:TimeInitSatVsClauseLength} compares the time needed for the initial \satCall{} on the $x$-axis with the average length of the \clauseCut{}s generated by our procedure for the instance on the $y$-axis, both for the \baseMethodShort{} as well as the \proofMethodShort{}.
The empty upper right area in the graph suggests that solution time for the initial \satCall{} gives some implicit upper bound on the length of the clauses found.
Hence, the experiments suggest that the harder the initial \satCall{}, the shorter and thus the stronger the \clauseCut{}s found by our procedures are, and thus the faster the subsequent optimization is, as shown in Section \ref{sec:study:speedup}.
\begin{figure}
    \centering
    \includegraphics[width=0.66\textwidth]{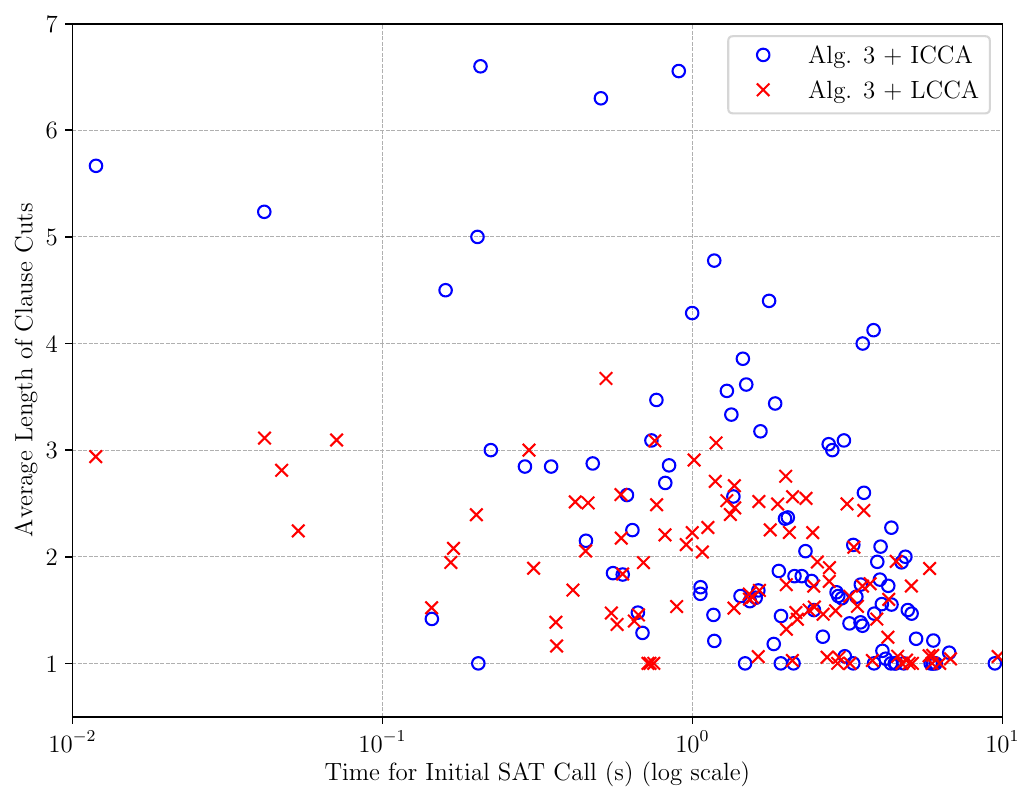}
    \caption{Connection between average clause length of \clauseCut{}s and time taken by the initial \sat{} call.}
    \label{fig:TimeInitSatVsClauseLength}
\end{figure}

\subsection{Influence of the Time needed for the Initial \satCall{} on the total Runtime}
\label{sec:study:initialVsTotal}
In some of the benchmark problems, the methods presented in this paper are able to make optimization over the CNF almost as fast as simply solving the associated \satProblem{}.
In these cases, the total runtime of the optimization is not significantly longer than the initial \sat-call which determines whether or not the problem is feasible at all.
We present the results in Figure \ref{fig:initialSatVsRuntime}.
Consistent with our findings from Section \ref{sec:study:initialVsClauseLength} (longer \satCall{}s lead to stronger \clauseCut{}s), the point cloud is sloped downwards.
This means that the longer it takes for the \satSolver{} to find a feasible assignment, the quicker the optimization can be done with our methods.
A theoretical limit of our methods would be reached if the whole optimization only takes as long as the initial \satCall{}, and indeed many instances come close to this limit. 
An instance on the limit would be on the dashed line.
\begin{figure}
    \includegraphics[width=0.66\textwidth]{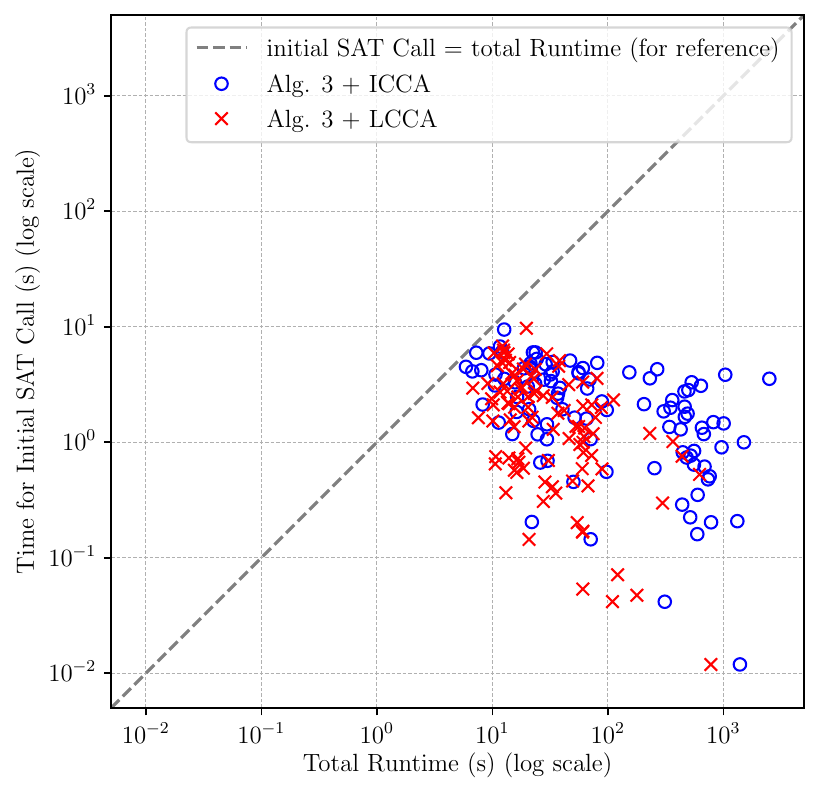}
    \centering
    \caption{Time taken by the initial \satCall{} vs. the total runtime of the optimization.}
    \label{fig:initialSatVsRuntime}
\end{figure}

\subsection{Node Feasibility in the MILP branch-and-bound tree and difficulty for \milpSolver{}s without \clauseCut{}s}
\label{sec:study:nodes}
To inquire about the impact of our cuts on the solution process, we analyzed the feasibility of nodes visited in the branch and bound tree. 
As \gurobi{} does not provide information to the user about the actual branching decisions, we examined the feasibility of a node by providing the integral values of variables in the node as assumptions to the \satSolver{} (akin to the assumptions in both the \baseMethodShort{} and \proofMethodShort{}).
The results are presented in Figure \ref{fig:nodeCount}.

Indeed, the performance of \gurobi{} is not optimal:
While usually visiting about $10^5$ and never less than $10^4$ nodes, the number of feasible nodes is low, finding only a single node with feasible integral values in 47\% of the instances.
This might contribute to the below-average performance of the general purpose \milpSolver{}.

On the other hand, our methods both substantially reduce the total number of nodes and increase the percentage of feasible nodes.
Applying our algorithms changes the MILP's behavior in most cases entirely, with the \proofMethodShort{} elevating the fraction of feasible nodes by 2 to 5 orders of magnitude and the \baseMethodShort{} achieving similar increases but mostly laying in between the results for \gurobi{} and the \proofMethodShort{}.
Moreover, the picture emerges that particularly hard instances are characterized by a high number of infeasible nodes even after the application of our algorithm.
Generally, it is remarkable that all but 13 of the 300 optimizations had less then 100 feasible nodes.
Hence it seems like reducing the number of infeasible nodes, i.e. identifying infeasibilities, is the key to quickly solve the benchmark instances.
\clauseCut{}s are particularly suited for this, as they can identify and eliminate all such infeasibilities. 

In many instances, adding our \clauseCut{}s eliminates all nodes but the root node for the problem.
This results in only a single node in the branch and bound tree of these instances.
This is the case in 1 instance for the \baseMethodShort{}, but 59 instances (equating to 59\% of the problem set) for the \proofMethodShort{}.
\gurobi{} alone does not achieve this in any instance.
In 20\% of the instances, the \proofMethodShort{} strengthens the relaxation so much that it already yields the integral solution to the problem. 
In the remaining cases, \gurobi{} either achieves an integral solution at the relaxation of the presolved model or adds further cuts itself through its own separation routines in addition to ours and thereby obtains an integral solution without branching.

\begin{figure}
    \includegraphics[width=0.66\textwidth]{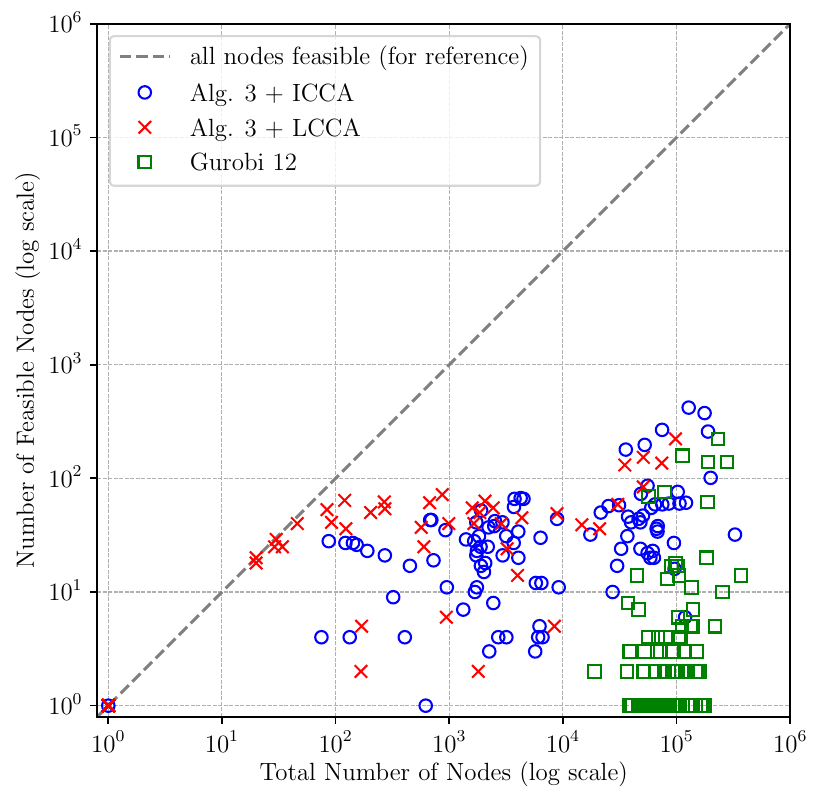}
    \centering
    \caption{The number of feasible nodes vs. the total number of nodes visited in the branch and bound optimization process.}
    \label{fig:nodeCount}
\end{figure}

\subsection{Summary of the computational study}
For the computational study, our separation algorithms \baseMethodShort{} and \proofMethodShort{} were used as a subroutine in Algorithm \ref{alg:solveInstance}.
In this context, both the \baseMethodShort{} and the \proofMethodShort{} significantly outperform the general purpose solver \gurobi{} 12 on the benchmark instances, by up to two orders of magnitude (Section \ref{sec:study:Gurobi}).
While \baseMethodShort{} is falling behind the specialized \maxsatSolver{} \rcTwo{} on harder instances, \proofMethodShort{} outperforms \rcTwo{}, requiring only about 60\% of the time needed by \rcTwo{} to solve the entire problem set (Section \ref{sec:study:RC2}).
Further, \rcTwo{} displayed sensitivity towards the magnitude of the objective coefficients, which is not the case for the MILP based approaches.
We provide cactus plots to compare all algorithms and illustrate the sensitivity of \rcTwo{} (Section \ref{sec:study:Cactus}).
The speedup of our methods against \gurobi{} is highly dependent on the strength of the \clauseCut{}s found, and high speedups were especially associated with clauses of length 1 and 2 (Section \ref{sec:study:speedup}).
In general, \cnfoptProblem{}s whose CNF required more time for the \satSolver{} to find any feasible solution led to stronger \clauseCut{}s (Section \ref{sec:study:initialVsClauseLength}).
Sometimes, the total time needed for our methods to solve the optimization problem was comparable to the time needed by the \satSolver{} to initially verify the feasibility of the problem, thus making optimization almost as fast as using the \satSolver{} itself in these cases (Section \ref{sec:study:initialVsTotal}).
An analysis of the feasibility of the nodes visited in the MILP's branch and bound trees (Section \ref{sec:study:nodes}) finally reveals that the problems' hardness seems to be caused by a very high number of infeasible nodes visited if \clauseCut{}s are not used. 
Conversely, \clauseCut{}s reduce the total number of nodes visited and elevate the share of feasible nodes by up to 5 orders of magnitude, very often leading to the problem being solved at the root node.

\section{Conclusion}
\label{sec:conclusion}
In summary, introducing \clauseCut{}s (strengthened \noGoodCut{}s specifically derived from CNF-based structures) into an \milp{}-framework yields substantial performance gains on randomly generated \maxsatProblem{}s compared to both the state-of-the-art general-purpose \milpSolver{} \gurobi{} and the specialized \maxsatSolver{} \rcTwo{}.
The two \clauseCut{}-separation algorithms proposed in this paper, the \baseMethod{} (\baseMethodShort{}) and the \proofMethod{} (\proofMethodShort{}), employ \satSolver{}s to derive and strengthen new clauses which are translated into cuts that cut off fractional points and can be seamlessly integrated into an \milp{}-framework. 
The \proofMethodShort{} in particular uses clauses learned by CDCL-\satSolver{}s, giving it both a theoretical and practical advantage to find more and stronger clauses than the \baseMethodShort{}, basically without any additional effort.

A computational study was conducted, where our algorithms were used in a cutting-plane procedure to strengthen the problems' \milp{}-formulation before they were given to the \milpSolver{}.
Using the \proofMethodShort{}, which was the best performing approach, total runtime (including the cutting-plane procedure) to solve the complete benchmark-set decreased to 7.8\% of the time taken by \gurobi{} and 60\% of the time taken by \rcTwo{}, with our strengthened relaxations producing integral optimal solutions in 20\% of the benchmark-instances and 59\% being solved at the root-node.

The analysis of the study suggests that the random benchmark-problems can be solved so efficiently by our methods, because the methods can find strong clauses implied by the problem's CNF.
Clauses of length 1 and 2 are having a specifically large impact and are crucial for high speedups.
Instances that take the \satSolver{} longer to solve tend to imply stronger clauses and can be optimized particularly fast.
In some cases, the initial \satSolver{} call accounts for nearly the entire runtime, making optimization effectively as fast as \sat{}-solving itself in these instances.

An inquiry into the feasibility of the nodes in the branch-and-bound tree revealed that without our cuts, almost all nodes visited by \gurobi{} have infeasible integral values at their LP-solution.
In fact, in 47\% of the instances \gurobi{} alone found only a single feasible node in the entire search-tree.
That means that in these instances, our cuts could have separated every but the final LP-solution in \gurobi{}'s branch-and-bound tree.
With our methods, whose cuts are capable of eliminating all such infeasibilities, the number of nodes visited is decreased by 2 to 5 orders of magnitude and the share of feasible nodes increased similarly as much.

We conclude that \clauseCut{}s are useful tools to improve the performance of \milp{}-based \maxsat{} optimizations and merit further research at the intersection of \milpLong{} and \maxsat{}.
Promising directions for future work include studying their impact on industrial instances with \cnfopt{} or \maxsat{} structure, as well as developing methods specialized in separating the particularly strong clauses of length one or two.

\section*{Acknowledgments}
\label{sec:acknowledgments}
The authors gratefully acknowledge the scientific support and HPC resources provided by the Erlangen National High Performance Computing Center (NHR@FAU) of the Friedrich-Alexander-Universität Erlangen-Nürnberg (FAU).
The hardware is funded by the German Research Foundation (DFG).

We also thank the authors of the \pysat{} package, whose toolkit was indispensable for our computational study.

\bibliographystyle{splncs04}
\bibliography{reference}

\end{document}